\documentclass[11pt]{amsart}
\usepackage{amssymb,amscd}
\usepackage{verbatim}

\let\frak\mathfrak

\def\>{\relax\ifmmode\mskip.666667\thinmuskip\relax\else\kern.111111em\fi}
\def\<{\relax\ifmmode\mskip-.333333\thinmuskip\relax\else\kern-.0555556em\fi}
\def\vsk#1>{\vskip#1\baselineskip}
\def\vv#1>{\vadjust{\vsk#1>}\ignorespaces}
\def\vvn#1>{\vadjust{\nobreak\vsk#1>\nobreak}\ignorespaces}

  \let\ssize\scriptstyle
\let\sssize\scriptscriptstyle

\let\Medskip\medskip
\def\medskip{\par\Medskip}
\let\Bigskip\bigskip
\def\bigskip{\par\Bigskip}

\let\Maketitle\maketitle
\def\maketitle{\Maketitle\thispagestyle{empty}\let\maketitle\empty}

\newtheorem{thm}{Theorem}[section]
\newtheorem{cor}[thm]{Corollary}
\newtheorem{lem}[thm]{Lemma}

\numberwithin{equation}{section}

\theoremstyle{definition}

\newtheorem*{rem}{Remark}

\let\mc\mathcal
\let\nc\newcommand

\let\al\alpha

\let\ka\kappa
\let\la\lambda

\let\phi\varphi
\let\si\sigma

\let\der\partial

\let\leq\leqslant

\let\on\operatorname
\let\bi\bibitem
\let\bs\boldsymbol

\def\C{{\mathbb C}}
\def\Z{{\mathbb Z}}

\def\F{{\mc F}}

\def\+#1{^{\{#1\}}}

\def\End{\on{End}}

\def\gln{\mathfrak{gl}_N}

\def\beq{\begin{equation}}
\def\eeq{\end{equation}}
\def\be{\begin{equation*}}
\def\ee{\end{equation*}}

\nc{\bea}{\begin{eqnarray*}}
\nc{\eea}{\end{eqnarray*}}
\nc{\bean}{\begin{eqnarray}}
\nc{\eean}{\end{eqnarray}}
\nc{\Ref}[1]{{\rm(\ref{#1})}}

\let\ga\gamma
\let\Ga\Gamma

\nc{\Il}{{\mc I_{\bs\la}}}
\nc{\bla}{{\bs\la}}
\nc{\Fla}{\F_\bla}
\nc{\tfl}{{T^*\Fla}}
\nc{\GL}{{GL_n(\C)}}
\nc{\GLC}{{GL_n(\C)\times\C^*}}

\let\sd s %% \def\sd{\dot s}

\def\ddk_#1{\kk_{#1}\<\>\frac\der{\der\<\>\kk_{#1}}}

\def\bul{\mathbin{\raise.2ex\hbox{$\sssize\bullet$}}}
\def\intt{\mathchoice
{\mathop{\raise.2ex\rlap{$\,\,\ssize\backslash$}{\intop}}\nolimits}
{\mathop{\raise.3ex\rlap{$\,\sssize\backslash$}{\intop}}\nolimits}
{\mathop{\raise.1ex\rlap{$\sssize\>\backslash$}{\intop}}\nolimits}
{\mathop{\rlap{$\sssize\<\>\backslash$}{\intop}}\nolimits}}

\let\kk q %% Q
\let\cc c

\let\Ko K

\def\GZ/{Gelfand-Zetlin}
\def\KZ/{{\slshape KZ\/}}
\def\qKZ/{{\slshape qKZ\/}}
\def\XXX/{{\slshape XXX\/}}

\nc{\A}{{\mc C}}

\let\8\infty
\gdef\){\>\]}
\gdef\){\RIfM@\mskip.333333\thinmuskip\relax\else\kern.0555556em\fi}
\gdef\]{{\!\!\;}}
\def\fd/{fin\-ite-dim\-en\-sion\-al}
\def\glk{{$\frak{gl}_k$}}
\def\gkmod/{\$\glk$-module}
\def\gnmod/{\$\gln$-module}

\def\Id{\on{Id}}

\begin{document}

\hrule width0pt
\vsk->

\title[On axioms of Frobenius like structure in the theory of arrangements]
{On axioms of Frobenius like structure in \\  the theory of arrangements}

\author
[Alexander Varchenko ]
{ Alexander Varchenko$\>^{\star}$}

\maketitle

\begin{center}
{\it Department of Mathematics, University of North Carolina
at Chapel Hill\\ Chapel Hill, NC 27599-3250, USA\/}
\end{center}

{\let\thefootnote\relax
\footnotetext{\vsk-.8>\noindent
$^\star$\,{\it E-mail}: anv@email.unc.edu,  supported in part by NSF grant DMS-1362924 and Simons Foundation grant \#336826.}}

\medskip

\begin{abstract}

A Frobenius manifold is a manifold with a flat metric and a Frobenius algebra structure on  tangent spaces 
at points of the manifold such that the structure constants of multiplication are given by third derivatives
of a potential function on the manifold with respect to  flat coordinates.

In this paper we present a modification of that notion coming from the theory of arrangements of hyperplanes.
Namely, given natural numbers $n>k$, we have a flat $n$-dimensional
manifold and a vector space $V$ with a nondegenerate symmetric bilinear form and an
algebra structure on $V$, depending on points of the manifold, such that the structure constants of  multiplication are given
by $2k+1$-st derivatives of a potential function on the manifold with respect to  flat coordinates.
We call such a structure a {\it Frobenius like structure}.   Such a structure arises when one has a family of arrangements of $n$ affine hyperplanes
in $\C^k$ depending on parameters so that the hyperplanes move parallely to themselves when the parameters
change. In that case a Frobenius like structure arises on the base $\C^n$ of the family.

\end{abstract}

\section{Introduction}
The theory of Frobenius manifolds has multiple connections with other branches of mathematics,
such as quantum cohomology, singularity theory, the theory of integrable systems, see, for example,
\cite{D1, D2, M}. 
The notion of a Frobenius manifold was introduced by B.\,Dubrovin in \cite{D1},   see also
 \cite{D2, M, FV,  St},  where numerous variants of this notion are discussed.  In all alterations, a
 Frobenius manifold is a manifold with a flat metric and a Frobenius algebra structure on  tangent spaces 
at points of the manifold such that the structure constants of multiplication are given by third derivatives
of a potential function on the manifold with respect to  flat coordinates.

In this paper we present a modification of that notion coming from the theory of arrangements of hyperplanes.
Namely, given natural numbers $n>k$, we have a flat $n$-dimensional
manifold and a vector space $V$ with a nondegenerate symmetric bilinear form and an
algebra structure on $V$, depending on points of the manifold, such that the structure constants of  multiplication are given
by $2k+1$-st derivatives of a potential function on the manifold with respect to  flat coordinates.
We call such a structure a {\it Frobenius like structure}.   Such a structure arises when one has a family of arrangements of $n$ affine hyperplanes
in $\C^k$ depending on parameters so that the hyperplanes move parallely to themselves when the parameters
change. In that case a Frobenius like structure arises on the base $\C^n$ of the family.
On such families see, for example,  \cite{MS, SV}.

\smallskip
In Section \ref{sec D Frob},  we give Dubrovin's definition of almost Frobenius structure. In Section \ref{sec Frob like}, we give the definition of
Frobenius like structure motivated by Dubrovin's definition and results of \cite{V4}. The definition consists of eight axioms: 
 \ref{Inv ax}, \ref{unit ax}, \ref{Comm}, \ref{pP ax}, \ref{ass ax}, \ref{Homogeneity axiom}, \ref{Lifting axiom}, \ref{Axiom e}.
 In Section \ref{sec ex}, we consider a family of arrangements and construct a Frobenius like structure on its base.

%!!!

\smallskip
The paper had been written while the author visited the MPI in Bonn in 2015-2016.
The author thanks the MPI for hospitality, Yu.I.\,Manin for interest in this work, 
P.\,Dunin-Barkowski, B.\,Dubrovin, and  C. Hertling for useful discussions.

\section{Axioms of an almost Frobenius structure, \cite{D2}}
\label{sec D Frob}

An {\it almost Frobenius structure} of charge $d\ne 1$ on a manifold $M$ is
a structure of a Frobenius algebra on the tangent spaces $T_zM  = ( *_z , ( \,,\, )_z)$, where $z\in M$.
 It satisfies the following axioms.
 
 \smallskip
 \noindent
 {\it  Axiom 1.} The metric $(\, ,\, )_z$ is flat.
 
  \smallskip
 \noindent
 {\it  Axiom 2.} 
 In the flat coordinates $z_1,\dots,z_n$ for the metric, the structure constants
 of the multiplication
 \bean
 \label{D mult}
 \frac \der{\der z_i}*_z\frac \der{\der z_j} = N^k_{i,j}(z) \frac \der{\der z_k}
 \eean
 can be represented in the form
 \bean
 \label{D pot}
 N^k_{i,j}(z) = G^{k,l}\frac{\der^3 L}{\der z_l\der z_i\der z_j}(z)
\eean
for some function  $L(z)$.  Here  $(G^{i,j})$ is the matrix inverse to the matrix  
$(G_{i,j}:=(\frac\der{\der z_i},\frac \der{\der z_j}))$.

  \smallskip
 \noindent
 {\it  Axiom 3.} 
The function $L(z)$ satisfies the homogeneity equation
\bean
\label{D homoge}
\sum_{j=1}^nz_j\frac{\der L}{\der z_j}(z)= 2 L(z) + \frac 1{1-d} \sum_{i,j} G_{i,j}z_iz_j.
\eean

  \smallskip
 \noindent
 {\it  Axiom 4.}  The Euler vector field
 \bean
 \label{Euler}
 E=\frac {1-d}2\sum_{j=1}^n z_j\frac\der{\der z_j}
 \eean
 is the unit of the Frobenius algebra.
 Notice that $(E,E) = \frac{(1-d)^2}4 \sum_{i,j}G_{i,j}z_iz_j$.

\smallskip
 \noindent
 {\it  Axiom 5.}  There exists a vector field
 \bean
 \label{ D vect e}
 e=\sum_{j=1}^n e_j(z)\frac\der{\der z_j},
 \eean
  an invertible 
 element of the Frobenius algebra $T_zM$ for every $z\in M$,
such that  the operator $g(z) \mapsto (e g)(z)$ 
acts by the shift  $\ka\mapsto \ka-1$ on the space of
solutions of the system of equations
 \bean
\label{D flat section}
\frac{\der ^2 g}{\der z_i\der z_j}=\ka \sum_{l=1}^n N^l_{i,j} \frac{\der g}{\der z_l},
\qquad i,j=1,\dots,n.
\eean
 The solutions of that system are called {\it twisted periods},  \cite{D2}.
 
 \subsection{Example in Section 5.2 of \cite{D2}} Let 
 \bean
 \label{D Pot}
 F(z_1,\dots,z_n) = \frac {n}4\sum_{i<j}(z_i-z_j)^2 \log (z_i-z_j)
 \eean
 and let the metric be $(\frac \der{\der z_i},\frac \der{\der z_j})=\delta_{i,j}$.
 The third derivatives of $F(z)$ give the multiplication law of tangent vectors
$\der_i:=\frac\der{\der z_i}$:
\bean
\label{D E mult}
\der_i *_z\der_ j
=
 -\frac{n}2\frac{\der_i-\der_j}{z_i-z_j} \quad 
 \on{for}\ i\ne j,
\qquad
\der_i *_z\der _i
=
-\sum_{j\ne i}\der_i *_z\der_ j.
\eean
The unit element of the algebra is the Euler vector field 
\bean
\label{D EE}
E=\frac 1{n} \sum_{j=1}^n z_j\der_j .
\eean
The operator of multiplication in the algebra by the element $\sum_{j=1}^n\frac\der{\der z_j}$ equals zero. Factorizing over this direction
one obtains an almost Frobenius structure on $\{ z_1+\dots+z_n=0\}-\cup_{i<j}\{z_i=z_j\}$.  

The vector field $e$ is
\bean
\label{D ee}
e=-\sum_{j=1}^n \frac 1{f'(z_j)}\der_j, \qquad
f(t):=\prod_{i=1}^n(t-z_i).
\eean
Equations \Ref{D flat section} take the form
\bean
\label{D E flat section}
\der_i\der_j g = -\frac \kappa{z_i-z_j} (\der_ig-\der_jg), \qquad i\ne j
\eean
with solutions 
\bean
\label{D solution}
g(z_1,\dots,z_n) = \int \prod_{i=1}^n (t-z_i)^\ka dt.
\eean

\section{Axioms of a Frobenius like structure}
\label{sec Frob like}

\subsection{Objects}
\label{sec obj}

Let  an $n\times k$-matrix $(b^j_i)^{j=1,\dots,k}_{i=1,\dots,n}$ be given with $n>k$.  
For $i_1,\dots,i_k\in [1,\dots,n]$, denote
\bean
\label{def d}
d_{i_1,\dots,i_k} = \det (b^j_{i_l})_{j,l=1,\dots,k}.
\eean
We have  $d_{\si i_1,\dots,\si i_k} = (-1)^\si d_{i_1,\dots,i_k}$ for $\si\in S_k$.
We assume that $d_{i_1,\dots,i_k}\ne 0$ for all distinct $i_1,\dots,i_k$.
\begin{rem}

The numbers $(d_{i_1,\dots,i_k})$ satisfy  the Pl\"ucker relations.
For example, for arbitrary $j_1$, \dots, $j_{k+1}$, $i_2$,
\dots, $i_{k}\in [1,\dots,n]$,
we have
\bean
\label{Pl rel}
\sum_{m=1}^{k+1} (-1)^{m+1} d_{j_1,\dots,\widehat{j_m},\dots,j_{k+1}} d_{j_m,i_2,\dots,i_{k}} =0.
\eean

\end{rem}

We assume that  an ${n-1\choose k}$-dimensional complex vector space $V$ 
is given with a non-degenerate symmetric bilinear form $S(\,,\,)$.

\smallskip

We assume that for every $k$-element unordered subset $\{i_1,\dots,i_k\}\subset [1,\dots,n]$ 
a nonzero vector $P_{i_1,\dots,i_k}\in V$ is given such that the vectors $\{P_{i_1,\dots,i_k}\}$ 
generate $V$ as a vector space and the only linear relations among them are
\bean
\label{pp rel}
\sum_{j=1}^n d_{j,i_2,\dots,i_k} P_{j, i_2,\dots,i_k}=0,
\qquad 
\forall i_2,\dots,i_k\in [1,\dots,n].
\eean

\begin{lem}
For any $j\in[1,\dots,n]$, the vectors $\{P_{i_1,\dots,i_k}\ |\  \{i_1,\dots,i_k\}\subset [1,\dots,n]-\{j\}\}$ form a basis of $V$.
\qed
\end{lem}

We assume that a ({\it potential}) function
 $L(z_1,\dots, z_n)$ is given.

\smallskip

We assume that nonzero vectors  $p_1(z_1,\dots, z_n),\dots,p_n(z_1,\dots, z_n) \in V$ depending on $z_1,\dots,z_n$
 are given  such that for any $i_2,\dots,i_{k}\in [1,\dots,n]$, we have
\bean
\label{p rel}
\sum_{j=1}^n d_{j,i_2,\dots,i_{k}} p_j(z_1,\dots, z_n) =0.
\eean

\subsection{Axioms}
\subsubsection{Invariance axiom } 
\label{Inv ax}

We assume that for any $i_2,\dots, i_{k}\in[1,\dots,n]$, we have
\bean
\label{L rel 1}
\sum_{j=1}^n d_{j, i_2,\dots,i_{k}}\frac{\der L}{\der z_j} =0.
\eean

Given $z=(z_1,\dots,z_n)$, for $j=1,\dots,n$, define a linear operator $p_j(z)*_z : V\to V$ by the formula
\bean
\label{mult}
S(p_j(z)*_z P_{i_1,\dots,i_k}, P_{j_{1},\dots,j_{k}}) = \frac{\der}{\der z_j} 
 \frac{\der^{2k}L}{\der z_{i_1}\dots\der z_{i_k}\der z_{j_1}\dots\der z_{j_{k}}}(z).
\eean

\begin{lem}
\label{lem corr}
The operators $p_j(z)*_z$ are well-defined.
\end{lem}

\begin{proof}
We need to check that if 
we substitute in \Ref{mult} a relation of the form 
\Ref{pp rel}  
instead of $P_{i_1,\dots,i_k}$ or $ P_{j_{1},\dots,j_{k}}$  or
if 
we substitute in \Ref{mult} a relation of the form \Ref{p rel} instead of $p_j(z)$, then the right-hand side of
\Ref{mult} is zero, but that follows from assumption \Ref{L rel 1}.
\end{proof}

\begin{lem}
\label{lem sym}
The operators $p_j(z)*_z$ are symmetric, $S(p_j(z)*_z v, w) = S(v,p_j(z)*_zw)$ for any $v,w\in V$.
\qed
\end{lem}

Choose a basis $\{P_\al\ |\ \al\in \mc I\}$  of $V$ among the vectors $\{P_{i_1,\dots,i_k}\}$.
Here each $\al$ is an unordered $k$-element subset $\{i_1,\dots,i_k\}$
of $[1,\dots,n]$ and $|\mc I|={n-1\choose k}$. Denote $\frac {\der^k}{\der^k z_\al} := \frac {\der^k}{\der z_{i_1}\dots\der z_{i_k}}$. 
Then \Ref{mult} can be written as
\bean
\label{mult basis}
S(p_j*_z P_\al, P_\beta) = \frac{\der}{\der z_j} 
 \frac{\der^{2k}L}{\der^k z_{\al}\der^k z_{\beta}}(z),
 \qquad \al,\beta \in \mc I.
 \eean
Denote $S_{\al,\beta} = S(P_\al,P_\beta)$. Let $(S^{\al,\beta})_{\al,\beta\in \mc I}$ be the matrix inverse to the matrix 
$(S_{\al,\beta})_{\al,\beta\in \mc I}$.
Introduce the matrix  $(M_{j,\al}^\ga)_{\al,\ga\in\mc I}$ of the operator $p_j(z)*_z$ by the formula
$p_j(z)*_z P_\al =\sum_{\ga\in\mc I} M^\ga_{j,\al}(z)P_\ga$. Then
\bean
\label{M j al ga}
M^\ga_{j,\al} =  \sum_{\beta\in\mc I}
 \frac{\der^{2k+1}L}{\der z_j\der^k z_{\al}\der^k z_{\beta}} S^{\beta, \ga} .
\eean

\subsubsection{Unit element axiom}
\label{unit ax}

We assume that for some $a\in\C^\times$, we have
\bean
\label{identity}
\frac 1a \sum_{j=1}^n z_j p_j(z)*_z = \Id \  \in \End(V),
\eean
that is, $\sum_{j=1}^n z_j M^\ga_{j,\al}(z) \,= \,a\,\delta_{\al}^\ga$.

\subsubsection{Commutativity axiom}
\label{Comm}
We assume that the operators $p_1(z)*_z,\dots,p_n(z)*_z$ commute.
In other words,  
\bean
\sum_{\beta} [M_{i,\al}^\beta M^\ga_{j,\beta}
- M_{j,\al}^\beta M^\ga_{i,\beta}]=0.
\eean
This is a quadratic relation between the $(2k+1)$-st derivatives of the potential.

\subsubsection{ Relation between $p_j(z)$ and $P_{i_1,\dots,i_k}$ axiom}
\label{pP ax}

 We assume that we have
\bean
\label{p=P}
p_{i_1}(z)*_z \dots *_z p_{i_k}(z)  = P_{i_1,\dots,i_k},  
\eean
for every $z$ and all unordered $k$-element subsets $\{i_1,\dots,i_k\}\subset [1,\dots,n]$.  

More precisely,
let $p_j(z)= \sum_{\al\in\mc I} C^\al_j(z) P_\al$, $j=1,\dots,n$, be the expansion 
of the vector $p_j(z)$ with respect to the basis $\{P_\al\}_{\al\in\mc I}$. Then
equation \Ref{p=P} is an equation on the coefficients $C^\al_{i_k}(z)$ and  $M_{i_m, \al}^\beta(z)$,
where $ \al,\beta\in\mc I$, $m=i_1,\dots,i_{k-1}$.
For example, if $P_{i_1,\dots,i_k}$ is one of the basis vectors $(P_\al)_{\al\in\mc I}$, then
\bean
\sum_{\al_2,\dots,\al_k\in\mc I} C^{\al_k}_{i_k} (z)M^{\al_{k-1}}_{i_{k-1},\al_k}(z)
M^{\al_{k-2}}_{i_{k-2},\al_{k-1}} (z)\dots
M^{\al_{1}}_{i_{1},\al_2}(z) = \delta^{\al_1}_{i_1,\dots,i_k}.
\eean

\begin{rem}
If $k=1$, then this axiom says that $p_j=P_j$ for $j\in[1,\dots,n]$. In this case the commutativity axiom follows from
\Ref{mult}.
\end{rem}

Define the multiplication on $V$ by the formula
\bean
\label{Mult}
P_{i_1,\dots,i_k}*_z P_{j_1,\dots,j_k} = p_{i_1}(z)*_zp_{i_2}(z)*_z\dots *_zp_{i_k}(z)*_z P_{j_1,\dots,j_k}.
\eean
The multiplication is well-defined due to relations \Ref{pp rel}, \Ref{p rel},  \Ref{L rel 1}. 

\begin{lem}
The multiplication on $V$ is commutative.

\end{lem}
\begin{proof}
 Indeed,
if the subsets $\{i_1,\dots,i_k\}$ and $\{j_1,\dots,j_k\}$ have a nonempty intersection, say $i_k=j_k$, then
\bea
&&
P_{i_1,\dots,i_{k-1}, i_k}*_z P_{j_1,\dots,j_{k-1},i_k} = (p_{i_1}*_z\dots *_z p_{i_{k-1}})*_zp_{i_k}*_z 
(p_{j_1}*_z \dots *_z p_{j_{k-1}})*_z p_{i_k}
= 
\\
&&
=(p_{j_1}*_z \dots *_z p_{j_{k-1}})*_z
  p_{i_k}*_z (p_{i_1}*_z\dots *_zp_{i_{k-1}})*_z  p_{i_k} =P_{j_1,\dots,j_{k-1}, i_k}*_z P_{i_1,\dots,i_{k-1},i_k}.
\eea
If the subsets $\{i_1,\dots,i_k\}$ and $\{j_1,\dots,j_k\}$ do not intersect, then
\bea
&&
P_{i_1,\dots,i_k}*_z P_{j_1,\dots,j_{k}} = p_{j_1}*_z \dots *_z p_{j_{k-1}}*_z 
p_{i_1} *_z(p_{i_2} \dots*_zp_{i_{k-1}} *_z p_{i_{k}}*_z p_{j_k})
= 
\\
&&
=p_{j_1}*_z \dots *_z p_{j_{k-1}}*_z p_{i_1} *_z P_{i_2,\dots,i_{k-1},i_{k},j_k}
=p_{j_1}*_z \dots *_z p_{j_{k-1}}*_z p_{i_1} *_z P_{i_2,\dots,i_{k-1},j_{k},i_k}=
\\
&&
=P_{j_1,\dots,j_k}*_z P_{i_1,\dots,i_{k}} .
\eea
\end{proof}

For $\al,\beta\in\mc I$, let  $P_\al*_zP_\beta = \sum_{\ga\in\mc I} N_{\al,\beta}^\ga(z) P_\ga$. If $\al=\{i_1,\dots,i_k\}$
then 
\bean
\label{N coeff}
N_{\al,\beta}^\ga(z) = \sum_{\beta_2,\dots,\beta_k\in\mc I}M^{\beta_k}_{i_k,\beta}(z)
M^{\beta_{k-1}}_{i_{k-1},\beta_{k}}(z)\dots M^{\beta_{2}}_{i_{2},\beta_{3}}(z)M^{\ga}_{i_{1},\beta_{2}}(z).
\eean
This is a polynomial of degree $k$ in the $(2k+1)$-st derivatives of the potential.

\subsubsection{Associativity axiom}
\label{ass ax}
 We assume that the multiplication on $V$ is associative, that is,
\bean\label{ass N}
\sum_{\ga\in\mc I}
[ N_{\al_2,\al_3}^\ga N_{\al_1,\ga}^\beta - N_{\al_1,\al_2}^\ga N_{\ga,\al_3}^\beta]=0.
\eean

\begin{thm}

For any $z$, the axioms \ref{Inv ax}, \ref{unit ax}, \ref{Comm}, \ref{pP ax}, \ref{ass ax} define on 
the vector space $V$ the structure $*_z$
of a commutative associative algebra with unit element $1_z= \frac 1a \sum_{j=1}^n z_j p_j(z)$.
The algebra is Frobenius, $S(u*_zv,w)= S(u, v*_zw)$, for all $u,v,w\in V$.
The elements $p_1(z),\dots,p_n(z)$ generate $(V, *_z)$ as an algebra.
\qed
\end{thm}

\subsubsection{Homogeneity axiom}
\label{Homogeneity axiom}
We assume that 
\bea
\sum_{j=1}^n z_j\frac\der{\der z_j} L \ =\  2k L \ +\   \frac{a^{2k+1}}{(2k)!}\, S(1_z,1_z),
\eea
where $a\in\C^\times$ is the same number as in the unit element axiom.

\subsubsection{Lifting axiom}
\label{Lifting axiom}
For $\ka\in\C^\times$, define on the trivial bundle $V\times \C^n\to\C^n$
a connection $\nabla^\ka$ by the formula:
\bean
\label{connection}
\nabla^\ka_j = \frac{\der}{\der z_j}\  -\ \ka\, p_j(z)*_z, \qquad  j=1,\dots,n.
\eean

\begin{lem}
\label{lem flat}
For any $\ka\in\C^\times$, the connection $\nabla^\ka$ is flat, that is, $[\nabla^\ka_i , \nabla^\ka_j]=0$ for all $i,j$. 
\end{lem}

\begin{proof}  The flatness for any $\ka$ is equivalent to the commutativity of the operators $p_i(z)*_z$, $p_j(z)*_z$
and the identity $\frac \der{\der z_j}( p_i(z)*_z) = \frac \der{\der z_i} (p_j(z)*_z)$. The commutativity  is our axiom \ref{Comm}
and the identity
follows from formula \Ref{M j al ga}.
\end{proof}

A flat section $s(z_1,\dots,z_n)\in V$ of the connection $\nabla^\ka$ is a solution of the
system of equations
\bean
\label{ka flat}
\frac{\der s}{\der z_j}\  -\ \ka\, p_j(z)*_z\, s=0,
\qquad
j=1,\dots,n.
\eean
We say that a flat section is {\it liftable}  if it has  the form
\bean
\label{spec sol}
s=\sum_{1\leq i_1<\dots<i_k\leq n} \frac {\der^k g}{\der z_{i_1}\dots\der z_{i_k}} \,d^2_{i_1,\dots,i_k}P_{i_1,\dots,i_k}
\eean
for a scalar function $g(z_1,\dots,z_n)$. If a section is liftable, the function $g(z)$ is not unique (an arbitrary polynomial
of degree less than $k$ can be added). Such a function $g$ will be called a {\it twisted period}.

We assume that for generic $\ka$ all flat sections of $\nabla^\ka$ are liftable and the set of all 
exceptional $\ka$ is contained in  a finite union of arithmetic progressions in $\C$.

\subsubsection{Differential operator $e$ axiom} 
\label{Axiom e}

 We assume that there exists a differential operator 
\bean
e=\sum_{1\leq i_1<\dots<i_k\leq n} e_{i_1,\dots,i_k}(z)\frac{\der^k}{\der z_{i_1}\dots\der z_{i_k}}
\eean
such that the operator
$g(z) \mapsto (e g)(z)$ acts as the shift $\ka\mapsto \ka-1$ on the space of solutions
of equations \Ref{ka flat}  and \Ref{spec sol}.  We also require that the element
\bean
\tilde e=\sum_{1\leq i_1<\dots<i_k\leq n} e_{i_1,\dots,i_k}(z)P_{i_1,\dots,i_k}\ \in\ V
\eean
is invertible for every $z$.

Notice that the operator $e$ annihilates all polynomials in $z$ of degree less than $k$.

\subsection{Remark}
The axioms of an almost Frobenius structure in Section \ref{sec D Frob} are similar to the axioms 
of a Frobenius like structure in Section \ref{sec Frob like} for $k=1$. 
The role of the tangent bundle $TM$ and the vectors $\frac \der{\der z_j}$
 in Section \ref{sec D Frob} is played by the trivial bundle $V\times \C^n\to\C^n$ and
 the vectors $p_j(z)\in V$ as well as the vectors $\frac \der{\der z_j}\in T\C^n$ in Section \ref{sec Frob like}.

\section{An example of a Frobenius like structure}
\label{sec ex}

\subsection{ Arrangements of hyperplanes.}
\label{arr}

 Consider a family of arrangements 
$\A(z) = \{H_i(z)\}_{i=1,\dots,n}$
of $n$ affine hyperplanes in $\C^k$. The arrangements of the family depend on parameters $z=(z_1,\dots, z_n)$.
When the parameters change, each hyperplane moves parallelly to itself.
More precisely, let $t=(t_1,\dots,t_k)$ be coordinates on $\C^k$.  For $i=1,\dots,n$, the hyperplane $H_i(z)$ is defined by the equation
\bea
f_i(t,z):= \sum_{j=1}^k b^j_it_j+z_i = 0 ,  \qquad \on{where} \ b^j_i\in\C  \  \on{are\ given}.
\eea
When $z_i$ changes the hyperplane $H_i(z)$ moves parallelly to itself.

We assume that every $k$ hyperplanes intersect transversally,  that is, for every distinct $i_1,\dots,i_k$, we assume that
\bea
d_{i_1,\dots,i_k}:={\det}_{l,j=1}^k (b^j_{i_l}) \ne 0.
\eea
For every $k+1$ distinct indices $i_1,\dots,i_{k+1}$, the corresponding hyperplanes have nonempty intersection if and only if
\bean
\label{f k+1}
f_{i_1,\dots,i_{k+1}}(z):= \sum_{l=1}^{k+1} (-1)^{l-1} d_{i_1,\dots,\widehat{i_l},\dots,i_{k+1}}z_{i_l}=0.
\eean

For any $1\leq i_1<\dots<i_{k+1}\leq n$ denote by $H_{i_1\dots,i_{k+1}}$ the hyperplane in $\C^n$ defined by the equation
$f_{i_1,\dots, i_{k+1}}(z)=0$.  The union
 $\Delta=\cup_{1\leq i_1<\dots<i_{k+1}\leq n} H_{i_1,\dots, i_{k+1}} \subset \C^n$ is called the {\it discriminant}.
We will consider only $z\in\C^n-\Delta$. In that case the arrangement  $\A(z)=\{H_i(z)\}_{i=1,\dots, n}$ has normal crossings.

\subsection{ Master function.} Fix complex numbers $a_1,\dots,a_n\in\C^\times$ such that $|a|:=\sum_{j=1}^na_j\ne 0$. 
The {\it master function} is the function
\bea
\Phi(t,z) = \sum_{j=1}^n a_j \log f_j(t,z) .
\eea
For fixed $z$, the master function is defined on the complement to our arrangement 
$\A(z)$ as a multivalued holomorphic function,
 whose derivatives are rational functions.
In particular, $\frac{\der \Phi}{\der z_i} = \frac {a_i}{f_i}$.
For fixed $z$, define the {\it critical set}
\bea
C_z: =\left\{t \in \C^k-\cup_{j=1}^n H_j(z) \ \Big|\ \frac{\der \Phi}{\der t_j}(t,z)=0,\  j=1,\dots,k \right\}
\eea
and the algebra of functions on the critical set
\bea
A_z: = \C(\C^k-\cup_{j=1}^n H_j(z))\Big/\left\langle \frac{\der \Phi}{\der t_j}(t,z),\  j=1,\dots,k\right\rangle ,
\eea
where $\C(\C^k-\cup_{j=1}^n H_j(z))$ is the algebra of regular functions on $\C^k-\cup_{j=1}^n H_j(z)$.
It is known that $\dim A_z = {n-1\choose k}$.  For example this follows from Theorem 2.4 and Lemmas 4.1, 4.2 in \cite{V3}.

Denote by $*_z$ the operation of multiplication in $A_z$.

The  {\it Grothendieck bilinear form} on $A_z$ is the form
\bea
(g,h)_z := \frac{1}{(2\pi i)^k}\int_{\Gamma}
\frac{gh\ dt_1\wedge\dots\wedge dt_k}{\prod_{j=1}^k \frac{\der \Phi}{\der t_j}}\quad
\on{for}\ g,h\in A_z.
\eea
Here  $\Gamma =\{t\ |\ |\frac{\der \Phi}{\der t_j}|=\epsilon,\ j=1,\dots,k\}$,
where  $\epsilon >0$ is  small. The cycle $\Ga$
is oriented 
 by the condition $d\arg \frac{\der \Phi}{\der t_1}\wedge\dots\wedge d\arg \frac{\der \Phi}{\der t_k} > 0$,
The pair $\big(A_z, (\,, \,)_z\big)$ is a Frobenius algebra.

\subsection{Algebra $A_z$}
For $i=1,\dots,n$, denote
\bean
\label{arr p}
p_i(z)= \Big[\frac{\der \Phi}{\der z_i}(t,z)\Big] = \Big[\frac{a_i}{f_i(t,z)}\Big]\ \in \ A_z,
\qquad i=1,\dots,n.
\eean
For any $i_1,\dots,i_k\in [1,\dots,n]$, denote
\bea
\label{P and w}
P_{i_1,\dots,i_k} := p_{i_1}(z)*_z\dots *_zp_{i_k}(z) \, \in  A_z\ \on{and}
\
w_{i_1,\dots,i_k}: = d_{i_1,\dots,i_k} p_{i_1}(z)*_z\dots *_z p_{i_k}(z) \,  \in  A_z.
\eea

\begin{thm} 
\label{thm p w ()}

${}$

\begin{enumerate}

\item[(i)]
 The elements $p_1(z),\dots, p_n(z)$ generate $A_z$ as an algebra and for any $i_2,\dots,i_k$ we have
 \bean
 \label{arr p rel}
 \sum_{j=1}^n d_{j, i_2,\dots,i_k} p_j(z)=0 .
 \eean
  
  \item[(ii)] The element 
  \bean
  \label{1 in A}
  1_z:=\frac1{|a|}\sum_{j=1}^nz_jp_j(z)
  \eean
   is the unit element of the algebra $A_z$.

\item[(iii)]
 The set of all  elements $w_{i_1,\dots,i_k}$ generates $A_z$ as a vector space.
The only linear relations between these elements are
\bean
\label{IS}
&&
w_{i_1,\dots,i_k} = (-1)^\si w_{\si i_1,\dots,\si i_k},   \quad \forall \si \in S_k,
\\
\notag
&&
\sum_{j=1}^n w_{j, i_2,\dots,i_k} = 0,\quad \qquad \forall\   i_2,\dots,i_k,
\eean
in particular, the linear relations between these elements do not depend on $z$.
\item[(iv)]
 The Grothendieck form has the following matrix elements:
\bean
\label{Mat elts}
&&
(w_{i_1,\dots,i_k}, w_{j_1,\dots,j_k})_z = 0, \qquad \on{if}\
|\{i_1,\dots,i_k\}\cap \{j_1,\dots,j_k\}|< k-1,
\\
&&
\notag
(w_{i_1,\dots,i_k}, w_{i_1,\dots,i_{k-1},i_{k+1}})_z =  \frac {(-1)^{k+1}}{|a|}\prod_{l=1}^{k+1} a_{i_l}
\quad\on{if} \ i_1,\dots,i_{k+1}\  \on{are\ distinct},
\\
\notag
&&
(w_{i_1,\dots,i_k}, w_{i_1,\dots,i_{k}})_z = \frac {(-1)^{k}}{|a|}\prod_{l=1}^k a_{i_l} \sum_{m\notin \{i_1,\dots,i_k\}} a_m 
\quad\on{if} \ i_1,\dots,i_{k}\  \on{are\ distinct}
\eean
\end{enumerate}
\end{thm}

\begin{proof} The statements of the theorem are the statements of  Lemmas 3.4, 6.2, 6.3, 6.6, 6.7 in \cite{V4}
transformed with the help of the 
isomorphism $\al$ of Theorem 2.7 in \cite{V4}, see also  Corollary 6.16 in \cite{V4} and Theorems  2.12 and 2.16 in \cite{V5}.
\end{proof}

This theorem shows that the elements $w_{i_1,\dots,i_k}$ define on $A_z$ a module structure over
$\Z$ independent of $z$.
With respect to this $\Z$-structure the Grothendieck form is constant. 

\begin{rem}
The presence of this $\Z$-structure in $A_z$ reflects the presence of the $\Z$-structure in the Orlik-Solomon algebra of the arrangement 
$\A(z)$, see \cite{V5, OS}.

\end{rem}

\begin{lem}
 Multiplication in $A_z$ is defined by the formulas
\bean
\label{mult in A}
p_{i_1}(z)*_z w_{i_2,\dots,i_{k+1}}
=
\frac{d_{i_2,\dots,i_{k+1}}}{f_{i_1, i_2,\dots,i_{k+1}}(z)}\,\sum_{\ell=1}^{k+1} (-1)^{\ell+1} a_{i_\ell} w_{i_1,\dots,\widehat{i_\ell},\dots,i_{k+1}},
\eean
if $ i_1\notin\{i_2,\dots,i_{k+1}\}$,
\bean
\label{mault in A 2}
p_{i_1}(z)*_z w_{i_1,i_2,\dots,i_k} =  - \sum_{m\notin\{i_1,\dots,i_k\}} p_{i_1}(z)*_z w_{m,i_2,\dots,i_k},
\eean
where $f_{i_1, i_2,\dots,i_{k+1}}(z)$ is defined in \Ref{f k+1}.

\end{lem}
\begin{proof} This is Lemma 6.8 in \cite{V4}.
\end{proof}

\begin{thm}
\label{thm id k}
For any $i_{0} \in [1,\dots,n]$, the unit element  $1_z \in A_z$ is given by the formula
\bean
\label{O}
1_z &=& \frac {1}{|a|^k} \sum_{1\leq i_1<\dots<i_k\leq n
\atop i_{0}\notin\{i_1,\dots,i_k\}} \frac{(f_{i_0,i_1,\dots,i_k}(z))^k}{\prod_{m=0}^{k} (-1)^md_{i_0,i_1,\dots,\widehat{i_m},\dots,i_{k}}} w_{i_1,\dots,i_k} .
\eean

\end{thm}

\begin{proof}  This is Theorem 6.12 in \cite{V4}.
\end{proof}

\subsection{Potential functions}
\label{sec potential function}

Denote
\bean
\label{1st kind}
Q(z) := (1_z,1_z)_z .
\eean
The function $(-1)^k Q(z)$ is called in \cite{V4} the {\it potential of first kind}. Theorem 3.11 in \cite{V4} says that for any
$i_1,\dots,i_k,j_1,\dots,j_k\in [1,\dots,n]$, we have
\bean
\label{pp der}
(p_{i_1}(z)*_z\dots *_zp_{i_k}(z), p_{j_1}(z)*_z\dots *_zp_{j_k}(z))_z\ =\
\frac{|a|^{2k}}{(2k)!} \,
\frac {\der^{2k} Q}{\der z_{i_1}\dots\der z_{i_k}\der z_{j_1}\dots\der z_{j_k}}.
\eean

\begin{thm}
\label{thm new}
We have
\bean
\label{(1,1)}
Q(z) = 
\frac {1}{|a|^{2k+1}} \sum_{1\leq i_1<\dots<i_{k+1}\leq n}\left(
\prod_{m=1}^{k+1}  \frac{a_{i_m}}{d_{i_1,\dots,\widehat{i_m},\dots,i_{k+1}}^2 }\right) (f_{i_1,\dots,i_{k+1}}(z))^{2k}.
\eean
\end{thm}

Cf. formulas (4.22) and (5.58) in \cite{V4}.

\begin{proof}  
Denote by $Q^\vee$ the right-hand side of \Ref{(1,1)}. Both $Q(z)$ and $Q^\vee(z)$ are homogeneous polynomials
of degree $2k$ 
in $z_1,\dots,z_n$. The $2k$-th derivatives of $Q(z)$ are given by formulas \Ref{pp der} and \Ref{IS}.
A direct verification shows that the  $2k$-th derivatives of $Q^\vee(z)$ also are given by the same
formulas  \Ref{IS}. Hence
$Q(z)=Q^\vee(z)$.
\end{proof}

Introduce the {\it potential function}
\bean
\label{POT2k}
&&
\\
\notag
&&
L(z_1,\dots,z_n) = \frac{1}{(2k)!} \sum_{1\leq i_1<\dots <i_{k+1}\leq n} \left(\prod_{m=1}^{k+1}
 \frac{a_{i_m}}{d_{i_1,\dots,\widehat{i_m},\dots,i_{k+1}}^2 }\right)
 (f_{i_1,\dots,i_{k+1}}(z))^{2k}\log (f_{i_1,\dots,i_{k+1}}(z)) .
\eean
In \cite{V4} this function was called the {\it potential  of second kind}.

\begin{thm}[\cite{V4}]
\label{thm 2k+1 der}
For any $i_0, i_1,\dots,i_k,j_1,\dots,j_k\in [1,\dots,n]$, we have
\bean
\label{LL der}
\phantom{aaa}
(p_{i_0}(z)*_zp_{i_1}(z)*_z\dots *_zp_{i_k}(z), p_{j_1}(z)*_z\dots *_zp_{j_k}(z))_z\ =\
\frac {\der^{2k+1} L}{\der z_{i_0}\der z_{i_1}\dots\der z_{i_k}\der z_{j_1}\dots\der z_{j_k}}.
\eean

\end{thm}
This is Theorem 6.20 in \cite{V4}.

\begin{lem}
\label{lem invariance}
For any $i_2,\dots, i_{k}\in[1,\dots,n]$, we have
\bean
\label{L rel 1 c}
\sum_{j=1}^n d_{j, i_2,\dots,i_{k}}\frac{\der L}{\der z_j} =0.
\eean
\end{lem}

\begin{proof} The lemma follows from the Pl\"ucker relations \Ref{Pl rel}.
\end{proof}

\begin{lem}
\label{lem homog}
We have
\bean
\label{homog}
\sum_{j=1}^n z_j\frac{\der L}{\der z_j}\ = \ 2k\, L\ +\  \frac{|a|^{2k+1}}{(2k)!}\, (1_z,1_z)_z.
\eean
\end{lem}

\begin{proof}
The lemma follows from the fact that each $f_{i_1,\dots,i_{k+1}}$ is a homogeneous polynomial in $z_1,\dots,z_n$ of degree one and from
Theorem \ref{thm new}.
\end{proof}

\subsection{Bundle of algebras $A_z$}

  The family  of algebras $A_z$ form a vector bundle over the space of parameters $z\in \C^n-\Delta$.
  The $\Z$-module
  structures of fibers canonically trivialize the vector bundle $\pi :\sqcup_{z\in \C^n-\Delta} A_z \to \C^n-\Delta$.

 Let $s(z) \in A_z$ be a section of that bundle,
\bean
\label{pres}
  s(z) = \sum_{i_1,\dots, i_k} s_{i_1,\dots, i_k}(z) w_{i_1,\dots, i_k},
\eean
where $s_{i_1,\dots, i_k}(z)$ are scalar functions. Define the {\it combinatorial connection} on the bundle of algebras
by the rule
\bean
\label{comb conn}
\frac{\der s}{\der z_j}: = \sum_{i_1,\dots, i_k} \frac{\der s_{i_1,\dots, i_k}}{\der z_j}  w_{i_1,\dots, i_k}, \quad \forall\ j.
\eean
 This derivative does not depend on the presentation \Ref{pres}. The combinatorial connection is {\it flat}.
The combinatorial connection is compatible with the  Grothendieck bilinear  form, that is
\bea
\frac {\der}{\der z_j}\left(s_1,\ s_2\right)_z\ =\ \left(\frac{\der}{\der z_j} s_1,\ s_2\right)_z 
\ +\  \left(s_1, \frac{\der}{\der z_j} s_2\right)_z
\eea
for any $j$ and  sections $s_1,s_2$.

 For $\ka\in\C^\times$, introduce a connection $\nabla^\ka$ on the bundle of algebras $\pi$ 
 by the formula
\bea
\nabla_j^\ka : = \frac{\der}{\der z_j} \ -\ \ka\, p_j (z)*_z, \qquad j=1,\dots, n.
\eea
This family of connections is a deformation of the combinatorial connection.

\begin{thm}
 For any $\ka\in\C^\times$, the  connection $\nabla^\ka$ is flat, 
\bea
[\nabla_j^\ka ,\nabla_m^\ka ]=0, \qquad\qquad \forall\ j,m.
\eea
\end{thm}
\begin{proof} This statement is a corollary of Theorems 5.1 and 5.9 in \cite{V3}.
\end{proof}

\subsection{Flat sections} Given $\ka\in\C^\times$, the flat sections
\bean
\label{section s}
s(z) = \sum_{1\leq i_1<\dots < i_k\leq n} s_{i_1,\dots, i_k}(z) w_{i_1,\dots, i_k}\
\in A_z
\eean
of the connection $\nabla^\ka$  are given by the following construction.
 
 The function $e^{\ka\Phi(t,z)} = (\prod_{i=1}^n f_i(t,z)^{a_i})^\ka$ defines a rank one local system $\mc L_\ka$ on 
 $U(\A(z)):=\C^k-\cup_{i=1}^n H_i(z)$  the complement to the arrangement $\A(z)$,
  whose horizontal sections over open subsets of  $U(z)$ are uni-valued branches of $e^{\ka\Phi(t,z)} $ 
 multiplied by complex numbers, see, for example, \cite{V2}.
 The vector bundle
\bean
\sqcup_{x\in \C^n-\Delta}\,H_k(U(\A(x)), \mc L_\kappa\vert_{U(\A(x))})
 \to  \C^n-\Delta,
 \eean
called the {\it homology bundle}, has the canonical  flat {\it Gauss-Manin connection}.

\begin{thm}
\label{thm hyper integrals}
Assume that $\ka$ is such that $\ka |a|\notin \Z_{\leq 0}$ and $\ka a_j\notin \Z_{\leq 0}$ for $j=1,\dots,n$,
then for every flat section $s(z)$, there  is a locally flat section
of the homology bundle  $\ga(z)\in H_k(U(\A(x))$, such that the coefficients $s_{i_1,\dots, i_k}(z)$ in \Ref{section s}
are given by the following formula:
\bean
\label{coeff of sec s}
s_{i_1,\dots, i_k}(z) = \int_{\ga(z)}e^{\ka\Phi(t,z)} d\log f_{i_1}\wedge\dots\wedge  d\log f_{i_k}, 
\qquad
\forall \ 
1\leq i_1<\dots < i_k\leq n .
\eean
\end{thm}

\begin{proof} This is Lemma 5.7 in \cite{V3}. The information on $\ka$ see in Theorem 1.1 in \cite{V1}.
\end{proof}

\smallskip

Theorem \ref{thm hyper integrals} identifies  the connection $\nabla^\ka$ with the Gauss-Manin connection
on the homology bundle. 
The flat sections are labeled by the locally constant cycles $\ga(z)$ in the complement to our arrangement.
The space of flat sections is identified with the degree $k$  homology group of the complement to the arrangement
$\A(z)$.
The monodromy of the flat sections is the monodromy of that  homology group.

\subsection{Special case}
Assume that $a_j=1$ for all $j\in [1,\dots,n]$. Then $|a|=n$.
Denote
\bean
\label{def g = 1}
g(z_1,\dots,z_n,\ka)
&=&
 \ka^{-k} \int_{\ga(z)}e^{\ka\Phi(t,z)} dt_1\wedge\dots\wedge  dt_{k} =
\\
\notag
&=&
\ka^{-k}
\int_{\ga(z)}\left(\prod_{j=1}^n f_j(t,z)\right)^\ka dt_1\wedge\dots\wedge  dt_{k}.
\eean
\begin{lem}
\label{lem lift}
If  $a_j=1$ for all $j\in [1,\dots,n]$, then the flat section of Theorem \ref{thm hyper integrals} can be written in the form
\bean
\label{spec sol example}
s(z)=\sum_{1\leq i_1<\dots<i_k\leq n} \frac {\der^k g(z,\ka)}{\der z_{i_1}\dots\der z_{i_k}} \,d^2_{i_1,\dots,i_k}P_{i_1,\dots,i_k},
\eean
cf. \Ref{spec sol}.
\qed
\end{lem}

\begin{cor}
\label{cor all sol special} If  $a_j=1$ for all $j\in [1,\dots,n]$ and  $\ka \notin\frac1n \Z_{\leq 0}$, then
 all flat sections of the connection $\nabla^\ka$ have the form \Ref{def g = 1},  \Ref{spec sol example}.

\end{cor}

\begin{thm}
\label{thm vector e}
Assume that $a_j=1$ for all $j\in [1,\dots,n]$, then 
 there exists a differential operator 
\bean
e=\sum_{1\leq i_1<\dots<i_k\leq n} e_{i_1,\dots,i_k}(z)\frac{\der^k}{\der z_{i_1}\dots\der z_{i_k}}
\eean
such that the operator
$g(z) \mapsto (e g)(z)$ acts as the shift $\ka\mapsto \ka-1$ on the space of all flat sections \Ref{def g = 1}, \Ref{spec sol example}.
Moreover, we have
\bean
\tilde e:= \sum_{1\leq i_1<\dots<i_k\leq n} e_{i_1,\dots,i_k}(z)P_{i_1,\dots,i_k} = \prod_{j=1}^n \Big[\frac1{f_j}\Big]\in A_z,
\eean
and the element $\tilde e$ is invertible in $A_z$.

\end{thm}

\begin{proof}
We have
\bea
g(z_1,\dots,z_n,\ka-1)=  
\int_{\ga(z)}\left(\prod_{j=1}^n f_j(t,z)\right)^\ka \frac 1 {\prod_{j=1}^n f_j(t,z)}
dt_1\wedge\dots\wedge  dt_{k}.
\eea
Hence it is enough to find the functions $e_{i_1,\dots,i_k}(z)$ such that
\bean
\label{coeff e}
 \frac 1 {\prod_{j=1}^n f_j(t,z)} = \sum_{1\leq i_1<\dots<i_k\leq n} e_{i_1,\dots,i_k}(z)\frac{1}
 {f_{i_1}(t,z)\dots f_{i_k}(t,z)}.
 \eean
Hence, for every  $z$, the coefficients must be of the form
\bean
\label{formula for a}
e_{i_1,\dots,i_k}  =  d_{i_1,\dots,i_k} \on{Res}_{f_{i_1}(t,z)=0}
\dots \on{Res}_{f_{i_k}(t,z)=0} \left(\frac {dt_1\wedge\dots\wedge dt_k} {\prod_{j=1}^n f_j(t,z)}
\right).
\eean
Conversely, if we choose such coefficients, then formula \Ref{coeff e} holds. Clearly, the element
$\tilde e= \prod_{j=1}^n \big[\frac1{f_j}\big]$ is invertible in $ A_z$.

\end{proof}

\subsection{Summary}
\label{sec summary}
Fix $a_1,\dots,a_n\in\C^\times$, such that $|a|\ne 1$.
Then the matrix $(b^i_j)$, algebra $A_z$, bilinear form $(\,,\,)_z$, vectors $p_j(z)$, vectors $P_{i_1,\dots,i_k}$,
element $1_z$, function $L(z_1,\dots,z_n)$ of Sections \ref{arr}-\ref{sec potential function} satisfy 
 axioms \ref{Inv ax}, \ref{unit ax}, \ref{Comm}, \ref{pP ax}, \ref{ass ax}, \ref{Homogeneity axiom}.

If in addition we assume that $a_j=1$ for all $j\in[1,\dots,n]$, then these objects also satisfy axioms \ref{Lifting axiom} and
\ref{Axiom e}.

\end{document}